\documentclass[12pt,twoside]{article}
 
\usepackage[margin=1in]{geometry} 
\usepackage{amsmath,amsthm,amssymb}
\usepackage{amsfonts}
\usepackage{hyperref}
\usepackage{fancyhdr}
\usepackage{graphicx}
\usepackage{tikz}
\usetikzlibrary{decorations.pathreplacing}
\newenvironment{theorem}[2][Theorem]{\begin{trivlist}
\item[\hskip \labelsep {\bfseries #1}\hskip \labelsep {\bfseries #2.}]}{\end{trivlist}}
\newenvironment{lemma}[2][Lemma]{\begin{trivlist}
\item[\hskip \labelsep {\bfseries #1}\hskip \labelsep {\bfseries #2.}]}{\end{trivlist}}

\newenvironment{corollary}[2][Corollary]{\begin{trivlist}
\item[\hskip \labelsep {\bfseries #1}\hskip \labelsep {\bfseries #2.}]}{\end{trivlist}}
\newenvironment{definition}[2][Definition]{\begin{trivlist}
\item[\hskip \labelsep {\bfseries #1}\hskip \labelsep {\bfseries #2.}]}{\end{trivlist}}
\newtheorem{remark}{Remark}[section]
\newtheorem{exampleone}{Example: The Hollow Square in $\mathbb{R}^2$}[section]
\newtheorem{exampletwo}{Example: Hausdorff Dimension of the Cantor Set}[section]
\newtheorem{exampleeem}{Example: Radial Projection}
\newtheorem{exampleeems}{A Familiar Example}
\newtheorem{op}{Open Problem}[section]
\begin{document}
 
 \newcommand{\keywords}[1]{%
  \begin{flushleft}
    \textbf{Keywords:} \small #1
  \end{flushleft}
}
\title{Geometric Analysis of Energy Minimizing Maps: Tangent Maps and Singularities} 
\author{
  Owen Drummond \\
  Department of Mathematics, Rutgers University \\
  \texttt{owen.drummond@rutgers.edu} \\
  \\
  \small{Advised by: Natasa Sesum} \\
}
 
\maketitle
\begin{abstract}
    Energy minimizing maps (E.M.M.'s) play a central role in the calculus of variations, partial differential equations (PDEs), and geometric analysis. These maps are often embedded into $C^\infty$ Riemannian manifolds to minimize the Dirichlet Energy functional under certain prescribed conditions. For understanding physical phenomena where systems naturally evolve to states of minimal energy, the geometric analysis of these maps has provided elucidating insights. This paper explores the geometric and analytic properties of energy minimizing maps, tangent maps, and the singular set (sing(u)). We begin by establishing key concepts from analysis, including the Sobolev Space $W^{1,2}$, harmonic functions, and Hausdorff dimension. Significant results about the density function, its upper semi-continuity, and the compactness theorem for tangent maps, and theorems for homogeneous degree zero minimizers are presented. Also analyzed in detail is the singular set (sing(u)), its Hausdorff dimension, and geometric structure. We conclude with open problems that are rich in research potential, and the far-reaching implications if these problems are to be solved. 

\end{abstract}

\tableofcontents

\section{Introduction}
Energy minimizing maps (E.M.M.'s) have emerged as a cornerstone in the fields of calculus of variations, partial differential equations (PDEs), and geometric analysis. E.M.M.'s find applications throughout various domains, including material science, fluid dynamics and general relativity, where the minimization of the Dirichlet energy functional plays a crucial role in predicting system behaviors. Historically, the calculus of variations has provided robust tools for identifying and analyzing minimizers of various functionals. The Dirichlet energy functional, central to the theory of harmonic maps, serves as a critical point for this study. Harmonic maps, which generalize harmonic functions to mappings between manifolds, highlight the importance of E.M.M.'s in revealing intricate geometric structures and advancing both pure and applied mathematics.

A significant challenge in the theory of E.M.M.'s is understanding the structure, geometric properties, and regularity of the singular set (sing(u)), the set of points at which the map $u$ fails to be smooth. Here, we define $u$ to be a map from an open subset $\Omega \subset \mathbb{R}^n$ to $N$, a $C^\infty$ Riemannian manifold embedded in some euclidean space $\mathbb{R}^p$. The analysis of sing(u), invloves exploring Hausdorff dimension and geometric characteristics, which lies at the cross section of Geometric Measure Theory. Tangent maps are critical in the local analysis of the singular set, revealing the local structure of singular points by convergence along a subsequence of rescaled energy minimizing maps. This involves studying the density function, which measures how the Dirichlet energy is distributed around these points, and understanding its upper semi-continuity. This paper provides a comprehensive exploration of the geometric and analytic properties of E.M.M.'s, focusing on tangent maps and the singular set. This paper will follow the structure of Leon Simon's \textit{Theorems on Regularity and Singularity of Energy Minimizing Maps}\cite{simon1996theorems}

\section{Analytic Preliminaries}
\subsection{The Sobolev Space $W^{1,2}(\Omega)$}

While a full exposition of Sobolev Spaces $W^{k,p}(\Omega)$ can be found in textbooks such as \textit{Partial Differential Equations} (2010, pp. 239-292)\cite{evans2010partial} by Lawrence Evans, here we are concerned with one Sobolev space in particular, which is \(W^{1,2}(\Omega)\), for \(\Omega \subset \mathbb{R}^n\) open. Here we define functions with \(L^2\) gradient:

\begin{definition}{1 (Functions with \(L^2\) Gradient)}
    We say that \(u \in L^2(\Omega)\) has a \textit{gradient in \(L^2\)} if there exist functions \(r_1, ..., r_n \in L^2(\Omega)\) such that
    \[
    \int_{\Omega} \phi r_i = - \int_{\Omega} u D_i \phi, \quad \forall \phi \in C_c^{\infty}(\Omega)
    \]
\end{definition}

Now note that if such $r_i$'s exist then they are unique, and if $u \in C^1(\Omega)$, then by integration by parts, the identity holds with $r_j=D_j u$, with $D_j$ being the usual partial derivative. Thus, the notion of functions with an $L^2$ gradient generalizes the notion of the classical derivative to include a larger space of functions (u need not be $C^1$), and we call such $r_j$ the $L^2$ \textit{weak derivatives} of $u$. Such functions $u$ belong to the Sobolev Space $W^{1,2}(\Omega)$. 

\begin{definition}{2 (Inner product on functions with an $L^2$ gradient)}
For $u,v$ functions with $L^2$ gradient, we define the inner product
\[
\langle u,v \rangle = \langle u,v \rangle_{L^2(\Omega)} + \sum_{j=1}^n \langle D_j u, D_j v \rangle_{L^2(\Omega)}
\]
and using the completeness of $L^2(\Omega)$, it is easy to verify that $W^{1,2}(\Omega)$ equipped with this inner product is a Hilbert Space. We will denote the typical inner product norm by $\lVert \cdot \rVert_{W^{1,2}(\Omega)}$.
    
\end{definition}

\subsection{Harmonic Functions}
\begin{definition}{3 (Harmonic Functions)}
    For $\Omega$ an open subset of $\mathbb{R}^n$, the function $u: \Omega \rightarrow \mathbb{R}$ is said to be \textit{harmonic} if $u \in C^2(\Omega)$ and $\Delta u=0$, where
    \[
    \Delta u = \sum_{j=1}^n D_j D_j u, \hspace{0.5 cm} D_j=\frac{\partial}{\partial x_j}
    \]
    note that we can also write $\Delta u = \textnormal{div}(Du)$, where $Du=(D_1 u, D_2 u, ..., D_n u)$ denotes the gradient of $u$.
\end{definition}

Harmonic functions exhibit a number of fascinating properties, one of those being the \textbf{mean value property}. Let $B_R(x_0)$ be a ball such that $\Bar{B_R(x_0)} \subset \Omega$, and by using divergence theorem, we obtain that
\[
u(x_0)=\frac{1}{\omega_n R^n} \int_{B_R(x_0)} u
\]
where $\omega_n$ denotes the volume of $B_1(0)$, the unit n-sphere. Multiplying this equation by $R^n$ and differentiating with respect to $R$ we obtain the second mean value property for $S^{n-1}$:
\[
u(x_0)=\frac{1}{\sigma_n R^{n-1}} \int_{S_R(x_0)} u
\]
where $\sigma_n = n\omega_n$ denotes the measure of $S^{n-1}=\partial B_1(0)$ and $S_R(x_0) = \partial B_R(x_0)$.
\subsection{Monotonicity Formula and Poincare Inequality}
\begin{theorem}{1 (The Monotonicity Formula)}
Let $0<\sigma<\rho<\rho_0$, then we have the important "monotonicity identity":
\[
\rho^{2-n} \int_{B_\rho (y)} |Du|^2-\sigma^{2-n} \int_{B_\sigma (y)} |Du|^2=2\int_{B_\rho (y) \setminus B_\sigma (y)} R^{2-n}|\frac{\partial u}{\partial R}|^2
\]
this formula holds provides $\bar{B}_{\rho_0}(y) \subset \Omega$. Here, $R=|x-y|$ and $\frac{\partial}{\partial R}$ is the directional derivative in the radial direction $\frac{(x-y)}{|x-y|}$. While I will not provide the proof of this identity here, there are many full proofs one can enjoy (ex. Simon). 
\end{theorem}
\begin{theorem}{2 (The Poincare Inequality)}
Let $\Omega \subset \mathbb{R}^n$ be a bounded and connected Lipschitz domain. There exists a constant $C_{\Omega}$ depending only on the choice of domain $\Omega$ such that for all $u \in W^{1,2}(\Omega)$
\[
\int_{\Omega} |u-\lambda|^2 \leq C_\Omega \int_{\Omega} |Du|^2
\]
holds, where $\lambda=|\Omega|^{-1} \int_{\Omega} u$. This is a critical tool from analysis and gives a bound for $u$ based on its gradient.
\end{theorem}

\subsection{Hausdorff Measure and Dimension}
\subsubsection{Hausdorff Measure}
Hausdorff Measure was introduced in Geometric Measure Theory as a generalization for length, area, or volume for sets that may be highly irregular and have non-integer dimension. While Lebesgue Measure captures this information for measurable sets in Euclidean Space, where it fails is with sets of interest in GMT, namely fractals, pathological null sets, and self-similar sets. As a result, Felix Hausdorff addressed this issue in 1918 with the invention of Hausdorff measure. This allows for the measurement of sets with non-integer dimension, as will be further explicated. Both Hausdorff Measure and Dimension have become fundamental tools in geometric measure theory for the characterization of rectifiable sets, fractals, and self similar sets.
The construction of Hausdorff Measure, for which you can see a complete treatment in \textit{Geometric Measure Theory}\cite{federer1996geometric} by Herbert Federer, is akin to the construction of Lebesgue Measure. One starts with a notion of outer measure and then restricts to sets that belong to the Carathéodory $\sigma$-algebra.

\begin{definition}{4 (n-dimensional Hausdorff Measure)}
The following also acts as an outer measure. Let $(X,\rho)$ be a metric space, and let $S \subset X$ be any subset, and $\delta > 0$. We first define
\[
H_\delta^d(S)=\inf \left\{\sum_{i=1}^{\infty} \left(\text{diam}(U_i)\right)^d : \bigcup_{i=1}^{\infty} U_i \supseteq S, \text{diam}U_i < \delta \right\}
\]
Notice that $H_\delta^d(S)$ is monotone nonincreasing in $\delta$, and hence $\lim_{\delta \to 0}$ exists, though it may be infinite. Thus we define the \textbf{n-dim Hausdorff Measure} to be 
\[
H^d(S)=\lim_{\delta \to 0} H_\delta^d(S) (=\sup_{\delta > 0} H_\delta^d(S))
\]
\end{definition}
Intuitively, the Hausdorff measure is capturing the total size of smaller and smaller open covers an arbitrary set $S$, which is where the outer measure structure is most apparent. However, the dimensionality of $H$ needs clarification:
\begin{exampleone}{1}
Consider the hollow unit square $S_h$ in $\mathbb{R}^2$. For $d=1$, $H^d(S_h)$ is the 'size' of $S_h$ with respect to the 1-dimensional case. That is, our open cover of $S_h$ becomes line segments, so $H^d$ is measuring the perimeter of $S_h$
\[
\implies H^d(S_h)=4
\]
However, $H^2(S_h)$ is measuring the \textit{area} of $S_h$ in $\mathbb{R}^2$. Because $S_h$ is hollow and composed of only line segments, these do not contribute any area, and hence
\[
\implies H^2(S_h)=0
\]
\end{exampleone}
\subsubsection{Hausdorff Dimension}
\begin{definition}{5 (Hausdorff Dimension)}
    Let $S \subset X$ as before. The Hausdorff dimension $\dim_{H}$ is given by
    \[
    \dim_{H}=\inf \left\{d \geq 0: H^d(X)=0\right\}
    \]
\end{definition}
Simply put, this is the smallest value of $d$ for which the Hausdorff Measure $H^d$ is identically zero. Let us proceed with an illuminating example.
\begin{exampletwo}{2}

\begin{center}
\begin{tikzpicture}
    \draw[thick] (0,0) -- (9,0);
    \node at (0,-0.3) {0};
    \node at (9,-0.3) {1};
    \node at (-0.5,0) {$C_0$};

    \draw[thick] (0,-1) -- (3,-1);
    \draw[thick] (6,-1) -- (9,-1);
    \node at (-0.5,-1) {$C_1$};

    \draw[thick] (0,-2) -- (1,-2);
    \draw[thick] (2,-2) -- (3,-2);
    \draw[thick] (6,-2) -- (7,-2);
    \draw[thick] (8,-2) -- (9,-2);
    \node at (-0.5,-2) {$C_2$};

    \draw[thick] (0,-3) -- (1/3,-3);
    \draw[thick] (2/3,-3) -- (1,-3);
    \draw[thick] (2,-3) -- (7/3,-3);
    \draw[thick] (8/3,-3) -- (3,-3);
    \draw[thick] (6,-3) -- (19/3,-3);
    \draw[thick] (20/3,-3) -- (7,-3);
    \draw[thick] (8,-3) -- (25/3,-3);
    \draw[thick] (26/3,-3) -- (9,-3);
    \node at (-0.5,-3) {$C_3$};

\end{tikzpicture}
\end{center}
\begin{center}
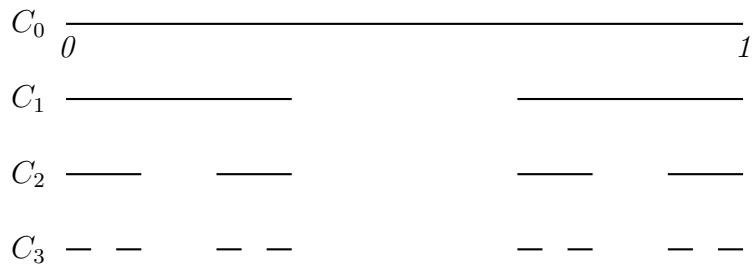
    Figure 1: Construction of the Cantor Set
\end{center}

Here, I will assume that the reader is already aware of the middle-thirds Cantor set $C \subset [0,1]$ given above, and thus I will not repeat its construction. For calculating the Hausdorff dimension of $C$, notice that after $n$ iterations, there are $2^n$ segments of length $\frac{1}{3^n}$. Thus, the total length of the segments is $\left(\frac{2}{3}\right)^n$. Since we are in $\mathbb{R}$, we are covering by line segments, and the diameter of each segment after $n$ steps is $\frac{1}{3^n}$. Thus,
\begin{align*}
H^d(C)=\lim_{\delta \to 0} \inf \left\{\sum_{i=1}^{\infty} \left(\text{diam}(U_i)\right)^d : \bigcup_{i=1}^{\infty} U_i \supseteq C, \text{diam}U_i < \delta \right\} \\
\implies H^d(C) = \lim_{n \rightarrow \infty} \sum_{i=1}^{2^n} \left(\frac{1}{3^n}\right)^d \implies H^d(C) = \lim_{n \rightarrow \infty} 2^n \left(\frac{1}{3^n}\right)^d = \lim_{n \rightarrow \infty} \left(2 \cdot 3^{-d}\right)^n
\end{align*}
and thus the Hausdorff dimension $d$ satisfies the equation
\[
0=H^d(C)=\lim_{n \rightarrow \infty} \left(2 \cdot 3^{-d}\right)^n
\]
and thus we require $2 \cdot 3^{-d} < 1$
\[
\implies \ln(2)+\ln\left(3^{-d}\right)<0 \implies \ln(2)-d\ln(3)<0 \implies d>\frac{\ln(2)}{\ln(3)}
\]
and thus \textbf{the Hausdorff dimension of the Cantor set is $\frac{\ln(2)}{\ln(3)} \approx 0.631$}. 
\end{exampletwo}
What is the significance of this value of 0.631? Well, intuitively, the Cantor Set is a null set, that is the Lebesgue Measure of $C$ is zero, but still contributes some space on the interval $[0,1]$. As the Hausdorff dimension of the unit segment $[0,1]$ is 1, while that of a single point would be 0. With $\dim_H(C) \approx 0.631$ the Cantor set's complexity in terms of its contribution to $[0,1]$ lies between that of a point and a line in $\mathbb{R}$.

\section{Energy Minimizing Maps}
\subsection{Energy Definitions}
First, we will establish some notation that will be used for the remainder of this paper. Suppose $\Omega$ is an open subset of $\mathbb{R}^n$, for $n \geq 2$, and that $N$ is a compact, $C^{\infty}$  Riemannian Manifold which is isometrically embedded into $\mathbb{R}^p$ for some $p$. We will look at maps $u: \Omega \rightarrow \mathbb{R}^p$ such that $u(\Omega) \subset N$. First, for $\Omega$ and $N$ defined in this way, we denote the Sobolev Space $W_{\text{loc}}^{1,2}(\Omega; N)$ to be the set of functions $u \in W_{\text{loc}}^{1,2}(\Omega; \mathbb{R}^p)$ with $u(x) \in N$ a.e. $x \in \Omega$. First we define a key function in the theory of harmonic maps:

\begin{definition}{6 (The Energy Functional)}
The energy $\mathcal{E}_{B_\rho (y)}(u)$ for a function $u \in W_{\text{loc}}^{1,2}(\Omega; N)$ in a ball $B_\rho (y) = \{ x: |x-y|<\rho \}$ with $\Bar{B}_{\rho} \subset \Omega$ is 
\[
\mathcal{E}_{B_\rho (y)}(u)=\int_{B_\rho (y)} |Du|^2
\]
\end{definition}
note that here $Du$ refers not to the gradient, but to the $n \times p$ matrix with entries given by $D_i u^j = \frac{\partial u^j}{\partial x^i}$, and $|Du|^2=\sum_{i=1}^n \sum_{j=1}^p (D_i u^j)^2$. 

\begin{definition}{7 (Energy Minimizing Maps)}
A map $u \in W^{1,2}(B_{\rho}(y); N)$ is called \textit{energy minimizing} if for each ball $B_\rho (y) \subset \Omega$
\[
\mathcal{E}_{B_\rho (y)}(u) \leq \mathcal{E}_{B_\rho (y)}(w) (\implies \int_{B_\rho (y)} |Du|^2 \leq \int_{B_\rho (y)} |Dw|^2)
\]
for every $w \in W^{1,2}(B_\rho (y); N)$ with $w=u$ in a neighborhood of $\partial B_{\rho}(y)$. 
\end{definition}
In fact, all harmonic maps are energy minimizing. Let us observe this through and example
\begin{exampleeem}
Define the map $u(x): \mathbb{R}^n \setminus \{ 0 \} \rightarrow S^{n-1}$ by
\[
u(x)=\frac{x}{|x|}
\]
it is easy to verify that $u$ is harmonic, but we will prove that $u$ is energy minimizing without appealing to this fact. To find the energy of $u$, given by $\mathcal{E}_{B_\rho (y)}(u)=\int_{B_\rho (y)} |Du|^2$, we must first find the square of the gradient $Du$:
\[
\frac{\partial u^i}{\partial x_j}=\frac{\delta_{ij}|x|-x_i \frac{x_j}{|x|}}{|x|^2}=\frac{\delta_{ij}}{|x|}-\frac{x_i x_j}{|x|^3}
\]
and thus
\[
|Du|^2=\sum_{i,j} (\frac{\partial u^i}{\partial x_j})^2 = \sum_{i,j} (\frac{\delta_{ij}}{|x|}-\frac{x_i x_j}{|x|^3})^2 = \frac{1}{|x|^2} (n-\frac{|x|^2}{|x|^2}) = \frac{n-1}{|x|^2}
\]
Now to find the total energy, we integrate away from the singularity at $x=0$ by integrating over the domain $\mathbb{R}^n \setminus B_\epsilon$, where $B_\epsilon$ is an epsilon ball centered at the origin.
\[
\implies \mathcal{E}_{\text{tot}}(u) = \int_{\mathbb{R}^n \setminus B_\epsilon} \frac{n-1}{|x|^2} dx
\]
To evaluate this integral, we convert to spherical coordinates, with $dx=r^{n-1} dr d\Omega$, so we have
\begin{align*}
\mathcal{E}_{\text{tot}}(u) &= \int_{S^{n-1}} d\Omega \int_{\epsilon}^{\infty} \frac{n-1}{r^2} r^{n-1} dr \\
&=(n-1) \int_{S^{n-1}} d\Omega \int_{\epsilon}^{\infty} r^{n-3} dr \\
&=(n-1)\text{Vol}(S^{n-1})([\frac{r^{n-2}}{n-2}]_{\epsilon}^{\infty}) \\
&=\lim_{r \to \infty}(r^{n-2}-\epsilon^{n-2}) \frac{(n-1)\text{Vol}(S^{n-1})}{n-2}
\end{align*}
Note that the integral is determined by the limit of $r \to \infty$, but $u$ is a map to a compact domain, and hence the overall energy is finite. By the finiteness of the energy integral, we can simply choose some scaling factor such that for a given $\delta > 0$, $\exists \lambda>0$ such that $u_\lambda(x)=\lambda \frac{x}{|x|}$, and hence
\[
\mathcal{E}_{\text{tot}}(u_\lambda)=\int_{\mathbb{R}^n \setminus B_\epsilon} |Du_\lambda|^2 dx < \delta
\]
and hence $u_\lambda$ is energy minimizing.
\vspace{1 cm}
\begin{center}
\begin{tikzpicture}

    \draw[thick] (0,0) circle (2);
    \node at (2.3, 0) {$S^1$};
    
    \fill (0,0) circle (2pt);
    \node at (0.2, 0.4) {0};

    \fill (3,1.5) circle (2pt);
    \node at (3.3, 1.5) {$x$};
    
    \draw[dashed, ->] (0,0) -- (3,1.5);
    \draw[->] (0,0) -- (1.6, 0.8);
    \node at (2, 1) {$\frac{x}{|x|}$};

    \fill (1.6, 0.8) circle (2pt);

    \node at (3.5, 1) {$\mathbb{R}^2 \setminus \{0\}$};

\end{tikzpicture} \\

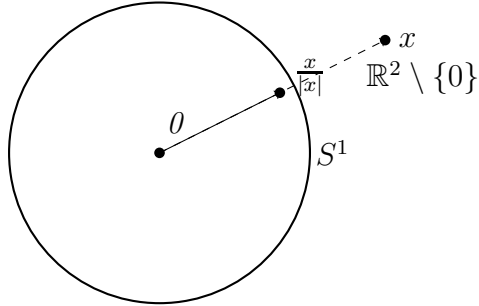
Figure 2: Radial projection from $\mathbb{R}^2 \setminus \{0\}$ to $S^1$

\end{center}

\end{exampleeem}
\begin{remark} 
Note that the Sobolev Space $W_{\text{loc}}^{1,2}(\Omega)$ is used in the definition for the energy functional, but $W^{1,2}(\Omega)$ is used for energy minimizing maps. For clarity, while the weaker (local) condition is sufficient for the definition of energy, analysts prefer to work with $W^{1,2}(\Omega)$ due to the regularity and completeness of the space which possesses more power in solving PDEs. From now on, you can forget about that ignominious space $W_{\text{loc}}^{1,2}(\Omega)$.
\end{remark}

There are many results about the regularity of energy minimizing maps. Since my interest lies in singularities as opposed to regularities, I will skip these results for now, though for further reference, I recommend Chapter 2 of Leon Simon's work (1996, pp. 19-51)\cite{simon1996theorems}. Key results include the $\epsilon$-regularity theorem, The Monotonicity Formula, Luckhaus' Lemma, and the Reverse Poincare Inequality.

\subsection{Density for E.M.M.'s}
\begin{definition}{8 (The Density Function)}
While defined in a slightly different manner than in the geometric measure theoretical sense, we define the \textit{density function} $\Theta_u: \Omega \rightarrow \mathbb{R}$ as
\[
\Theta_u(y)=\lim_{\rho \downarrow 0} \rho^{2-n} \int_{B_{\rho}(y)} |Du|^2 
\] and this limit exists for each point of $\Omega$ by the monotonicity formula.
\end{definition}

Now an important property of $\Theta_u$:

\begin{theorem}{3}
    $\Theta_u$ is upper semi-continuous on $\Omega$, that is 
    \[
    y_j \to y \in \Omega \implies \Theta_u(y) \geq \limsup_{j \to \infty} \Theta_u(y_j)
    \]
\end{theorem}
\begin{proof}
Let $\epsilon > 0$ and $\rho > 0$ with $\rho + \epsilon < \text{dist}(y, \partial \Omega)$. By the monoticity formula, we have that 
\[
\Theta_u(y_j) \leq \rho^{2-n} \int_{B_\rho (y_j)} |Du|^2
\]
for $j$ sufficiently large to ensure $\rho < \text{dist}(y_j, \partial \Omega)$. Now since $B_\rho (y_j) \subset B_{\rho+\epsilon} (y)$ for all $j$ large enough, we have that
\[
\Theta_u(y_j) \leq \rho^{2-n} \int_{B_{\rho+\epsilon} (y)} |Du|^2
\]
for sufficiently large $j$, so we obtain that $\limsup_{j \to \infty} \Theta_u(y_j) \leq \rho^{2-n} \int_{B_{\rho+\epsilon} (y)} |Du|^2$. By letting $\epsilon \downarrow 0$, we conclude that
\[
\limsup_{j \to \infty} \Theta_u(y_j) \leq \rho^{2-n} \int_{B_{\rho} (y)} |Du|^2 \implies \limsup_{j \to \infty} \Theta_u(y_j) \leq \Theta_u(y)
\] after taking the limit $\rho \downarrow 0$.
\end{proof}

The intuition behind the density is that it is a kind of averaging function, but for the energy functional rather than the function $u$. For example, for a function $f: \mathbb{R} \rightarrow \mathbb{R}$, think about the limit $\lim_{\rho \to 0} \frac{1}{\rho} \int_{B_{\rho} (y)} f d\lambda$, where $\lambda$ is the Lebesgue Measure on $\mathbb{R}$. By the Lebesgue Differentiation theorem, this value converges almost everywhere to $f(y)$. Similarly, density is an average value representation of the energy functional, though the a.e. convergence is not guaranteed due to the fact that $u$ is only $W^{1,2}(\Omega)$.

\subsection{Definition of reg(u) and sing(u)}
\begin{definition}{9 (Regular and Singular Sets)}
If $u \in W^{1,2}(\Omega: \mathbb{R}^p)$, then 
\[
\text{reg}(u) := \{ x \in \Omega : u \text{ is } C^{\infty} \text{ in a neighborhood of } x \}
\]
is the regular set of $u$, and 
\[
\text{sing}(u):=\Omega \setminus \text{reg}(u)
\]
is the singular set of $u$.
\end{definition}
\begin{exampleeems}
Again we consider the radial projection map, this time with dimension three or higher (otherwise sing(u) is empty, as we will see in Section 3). For $(x,y) \in \mathbb{R}^3 \times \mathbb{R}^{n-3}$, we define the map $u: \mathbb{R}^n \rightarrow S^2$ to be
\[
u(x,y)=\frac{x}{|x|}
\]
and we see that while obviously $0 \in \text{sing}(u) \implies \text{sing}(u) \neq \emptyset$, though there may be other nonzero singularities of $u$. In this case, due to the smoothness of radial projection, 0 is the only singularity of u, and hence
\[
\text{reg}(u) = \mathbb{R}^n \setminus \{0\}
\]
\end{exampleeems}
From this we obtain two corollaries immediately:
\begin{corollary}{1}
    There exists $\epsilon > 0$, depending only on $n, N$ such that if $B_{\rho}(y) \subset \Omega$ and if $\rho^{2-n} \int_{B_{\rho} (y)} |Du|^2 < \epsilon$, then $y \in \text{reg}(u)$ and $\sup_{B_{\rho/2} (y)} \rho^j |D^j u| \leq C$ for each $j=0,1,2,...,$ where $C$ depends only on $j,n,N$.
\end{corollary}
\begin{proof}
    The Poincare Inequality tells us that
    \[
    \inf_{\lambda \in \mathbb{R}^p} \rho^{-n} \int_{B_\rho (y)} |u-\lambda|^2 \leq C\rho^{2-n} \int_{B_{\rho} (y)} |Du|^2 \leq C \epsilon
    \]
    directly.
\end{proof}
\begin{corollary}{2}
    $\Theta_{u}(y)=0 \iff y \in \text{reg}(u)$
\end{corollary}
\begin{proof}
    The '$\impliedby$' direction follows from the fact that if $u$ is smooth in a neighborhood of $y$ ,the $|Du|$ is bounded near $y$ that is 
    \[
    |Du| < \delta
    \]
    for all $\delta>0$ in a neighborhood of $y$ and hence
    \[
    \Theta_u(y)=\lim_{\rho \downarrow 0} \rho^{2-n} \int_{B_{\rho}(y)} |Du|^2 < \epsilon
    \]
    for all $\epsilon > 0$ for the appropriate choice of $\delta$, and taking $\epsilon \downarrow 0$ we have
    \[
    \Theta_u(y)=0
    \]
    The '$\implies$' direction is a consequence of Corollary 1, that is
    \[
    \Theta_u(y)=0 \implies \lim_{\rho \downarrow 0} \rho^{2-n} \int_{B_{\rho}(y)} |Du|^2 < \epsilon 
    \]
    for all $\epsilon >0$, so $y \in \text{reg}(u)$ naturally follows.
\end{proof}
    
One natural question that arises is: what are the geometric properties of reg(u) and sing(u)? Moreover, how do these sets vary depending on the choice of $u$ as an energy minimizing map. As a result, sing(u) is of great interests to geometric analysts due to its seemingly recondite properties. Topics of interest pertaining to sing(u) include its smoothness/real-analyticity, Hausdorff Measure, and the case where $u$ embeds onto some n-sphere (see \textit{Singular Sets of Energy Minimizing Maps} by Fanghua Lin (1991, p. 165)\cite{lin1991singular}).

\subsection{Tangent Maps}
First, let us begin with a lemma about compactness:
\begin{lemma}{1 (The Compactness Theorem)}
    If $\{ u_j \}$ is a sequence of energy minimizing maps in $W^{1,2}(\Omega; N)$ with $\sup_j \int_{B_\rho (Y)} |Du_j|^2 < \infty$ for each ball $B_\rho (Y)$ with $\bar{B}_{\rho}(Y) \subset \Omega$, then there exists a subsequence $\{ u_{j'} \}$ and a minimizing harmonic map $u \in W^{1,2}(\Omega; N)$ such that $u_{j'} \rightarrow u$ in $W^{1,2}(B_\rho (y); \mathbb{R}^p)$ on each ball $\bar{B}_{\rho}(y) \subset \Omega$
\end{lemma}

The proof of this lemma is quite long and will not be given here.

Given $u:\Omega \rightarrow \mathbb{R}^p$ and $B_{\rho_0}(y)$ such that $\bar{B}_{\rho_0}(y) \subset \Omega$, and for any $\rho>0$, consider the scaling function $u_{y,\rho}$ given by 
\[
u_{y,\rho}(x)=u(y+\rho x)
\]
Note that on $B_{\rho_0}(0)$, $u_{y,\rho}$ is well defined. For $\sigma > 0$ and $\rho < \frac{\rho_0}{\sigma}$, after making a change of variables with $\Tilde{x}=y+\rho x$ in the energy integral for $u_{y,\rho}$ and noting that $Du_{y,\rho}(x)=\rho(Du)(y+\rho x)$, we have
\begin{align}
\sigma^{2-n} \int_{B_\sigma(0)} |Du_{y,\rho}|^2 = (\sigma \rho)^{2-n} \int_{B_{\sigma \rho}(y)} |Du|^2 \leq \rho_0^{2-n} \int_{B_{\rho_0}(y)} |Du|^2
\end{align}
by the Monotonicity Formula (1.3). Therefore, if $\rho_j \downarrow 0$, then $\limsup_{j \rightarrow \infty} \int_{B_{\sigma}(0)} |Du_{y,\rho_j}|^2 < \infty$ for all $\sigma > 0$, and so by the Compactness Theorem, there is a subsequence $\rho_{j'}$ such that $u_{y,\rho_{j'}} \rightarrow \phi$ locally in $\mathbb{R}^n$ w.r.t. the $W^{1,2}$-norm.

\begin{definition}{10 (Tangent Map)}
    Any $\phi$ obtained this way is called a \textit{tangent map of u at y}. Moreover, $\phi: \mathbb{R}^n \rightarrow N$ is an energy minimizing map with $\Omega = \mathbb{R}^n$.
\end{definition}
Now we will discuss some of the properties of Tangent maps. 

Let $\rho_j \downarrow 0$ be a sequence such that the scaled map $u_{y,\rho_j} \rightarrow \phi$. Then $u_{y,\rho_j}$ converges in energy to $\phi$, so after setting $\rho=\rho_j$ and taking limits on both sides of (1) as $j \rightarrow \infty$,
\[
\sigma^{2-n} \int_{B_\sigma(0)} |D\phi|^2 = \Theta_u(y)
\]
where $\Theta_u(y)$ is the density function, since definitionally, $\Theta_u(y)=\lim_{\rho \downarrow 0} \rho^{2-n} \int_{B_{\rho}(y)} |Du|^2$. In fact, we have that $\sigma^{2-n} \int_{B_\sigma(0)} |D\phi|^2$ is a constant function of $\sigma$, and by definition $\Theta_{\phi}(0)=\lim_{\sigma \downarrow 0} \sigma^{2-n} \int_{B_\sigma (0)} |Du|^2$, so
\[
\Theta_u(y)=\Theta_\phi (0)= \int_{B_\sigma (0)} |Du|^2 \hspace{0.5 cm} \forall \sigma>0
\]
Therefore any tangent map of $u$ at $y$ has a constant scaled energy which is equal to the density of $u$ at $y$.
Applying the Monotonicity Formula (1.3) to $\phi$, we have 
\begin{align*}
0=\sigma^{2-n}\int_{B_\sigma (0)}|D\phi|^2-\tau^{2-n}\int_{B_\tau (0)}|D\phi|^2 &= \int_{B_\sigma (0) \setminus B_\tau (0)} R^{2-n} |\frac{\partial \phi}{\partial R}|^2 \\
\implies \frac{\partial \phi}{\partial R} &= 0 \hspace{0.25 cm} a.e.
\end{align*}
Now since $\phi \in W_{\text{loc}}^{1,2}(\mathbb{R}^n;\mathbb{R}^p)$, we obtain by integration along rays in $\mathbb{R}^n$ the key property of tangent maps that
\[
\phi(\lambda x)=\phi(x) \hspace{0.5cm} \forall \lambda>0, x \in \mathbb{R}^n
\]

The next section will explore further properties of tangent maps as homogeneous degree zero minimizers. 

\begin{corollary}{1}
    $y \in \text{reg}(u) \iff \exists$ a constant tangent map $\phi$ of u at y
\end{corollary}
\begin{proof}
    By corollary 2 of Section 2.3:
    \[
    y \in \text{reg} \iff \Theta_u (y)=0 \iff \Theta_\phi(0)=0 \iff \phi \equiv \alpha
    \]
    for some constant $\alpha$.
\end{proof}

\subsection{Homogeneous Degree Zero Minimizers}
Define $\phi: \mathbb{R}^n \rightarrow N$ to be a degree zero minimizer, that is, $\phi(x)=\phi(\lambda x)$ for all $\lambda>0$, $x \in \mathbb{R}^n$

\begin{theorem}{4 (Maximum of the Density Function for D.Z.M.'s)}
    $\Theta_\phi (y)$ attains its maximum at $y=0$.
\end{theorem}
\begin{proof}
    By the monotonicity formula, for each $\rho>0$ and $y \in \mathbb{R}^n$
    \[
    2\int_{B_\rho (y)} R_y^{2-n} |\frac{\partial \phi}{\partial R_y}|^2+\Theta_\phi (y)=\rho^{2-n} \int_{B_\rho (y)} |D \phi|^2
    \]
    Here, $R_y(x)=|x-y|$ and $\frac{\partial}{\partial R_y}=|x-y|^{-1}(x-y)D$. Given that $B_\rho (y) \subset B_{\rho+|y|}(0)$, so that
    \begin{align*}
\int_{B_\rho(y)} |D\phi|^2 &\leq \rho^{2-n} \int_{B_{\rho+|y|}(0)} |D\phi|^2 \\
&= \left(1+\frac{|y|}{\rho}\right)^{n-2} (\rho+|y|)^{2-n} \int_{B_{\rho+|y|}(0)} |D\phi|^2 \\
&= \left(1+\frac{|y|}{\rho}\right)^{n-2} \Theta_\phi(0),
\end{align*}
since $\phi$ is homogeneous of degree zero, which ensures that $\tau^{2-n}\int_{B_\tau (0)}|D\phi|^2=\Theta_{\phi}(0)$. Therefore, letting $\rho \uparrow \infty$, we get that
\[
2 \int_{\mathbb{R}^n} R_y^{2-n} \left| \frac{\partial\phi}{\partial R_y} \right|^2 + \Theta_\phi(y) \leq \Theta_\phi(0)
\]
and so we have the desired inequality
\[
\Theta_{\phi}(y) \leq \Theta_{\phi}(0)
\]
and so the maximum of $\Theta_\phi$ is attained at $y=0$.
\end{proof}
Note that equality is established when $\frac{\partial\phi}{\partial R_y}=0$, that is $\phi(y+\lambda x)=\phi(y+x)$ for all $\lambda>0$
\begin{align*}
\phi(x) &= \phi(\lambda x) = \phi(y + (\lambda x - y)) = \phi(y + \lambda^{-1}(\lambda x - y)) \\
&= \phi(\lambda(y + \lambda^{-1}(\lambda x - y))) = \phi(x + y),
\end{align*}
where $t=\lambda-\lambda^{-1}$ is some arbitrary real number. Now we define the set 
\[
S(\phi)=\{y \in \mathbb{R}^n: \Theta_\phi (y)=\Theta_\phi (0) \}
\]

This gives way to the following theorem
\begin{theorem}{5 (The set $S(\phi)$ as a linear subspace)}
    $S(\phi)$ is a linear subspace of $\mathbb{R}^n$ and $\phi(x+y)=\phi(x)$ for $x \in \mathbb{R}^n$ and $y \in S(\phi)$
\end{theorem}
\begin{proof}
    For $\phi$ a degree zero minimizer, we have that $\phi(x)=\phi(x+ty)$ for all $x \in \mathbb{R}^n$, $t \in \mathbb{R}$ and $y \in S(\phi)$. From this, we obtain
    \[
    \phi(x+az_1+bz_2)=\phi(x)
    \]
    for all $a,b \in \mathbb{R}$ and $z_1, z_2 \in S(\phi)$. However, if $z \in \mathbb{R}^n$ and $\phi(x+z)=\phi(x)$ for every $x \in \mathbb{R}^n$, then $\Theta_\phi (z)=\Theta_\phi(0)$ by taking $x$ to be 0, and hence $z \in S(\phi)$
\end{proof}
For this, we have the immediate corollary:
\begin{corollary}{2}
    $\dim S(\phi)=n \iff S(\phi)=\mathbb{R}^n \iff \phi = \text{const}$
\end{corollary}
Moreover, a nonconstant homogeneous degree zero map cannot be continuous at 0, and so we always have $0 \in \text{sing}(u)$. Since $\phi$ is assumed to be nonconstant, and $\phi(x+z)=\phi(x)$ for any $z \in S(\phi)$, we have
\[
S(\phi) \subset \text{sing}(u)
\]
for $\phi$ nonconstant D.Z.M.

\section{The Properties and Geometric Picture of sing(u)}
\subsection{A Salient Theorem}
Here we will unveil what is perhaps the foremost geometric property of sing(u).
\begin{theorem}{6 (sing(u) Consists Only of Isolated Points)}
    If $u \in W^{1,2}(\Omega)$ is energy minimizing in $\Omega$, then $H^{n-2}(\text{sing}(u))=0$ (in particular, sing(u)=$\emptyset$ if $n=2$)
\end{theorem}
\begin{proof}
    Let $K \subset \Omega$ be compact, $\delta_0 < \text{dist}(K, \partial \Omega)$. For $y \in \text{sing}(u) \cap K$, we have 
    \begin{align}
        \int_{B_\rho (y)} |Du|^2 \geq \epsilon \rho^{n-2} \hspace{0.5 cm} \forall \rho < \delta_0
    \end{align}
    by Corollary 1 of the previous section. For fixed $\delta < \delta_0$, choose a maximal pairwise disjoint collection of balls $\{B_{\delta/2}(y_j)\}_{j=1,...,n}$ with $y_j \in \text{sing}(u) \cap K$ so that for $j=1,...J$, $J$ is the maximum integer for which the collection exists. This is a cover of $\text{sing}(u) \cap K$, because if there is some $z \in (K \cap \text{sing}(u)) \setminus (\cup_{i=1}^J B_{\delta/2}(y_j)$, this would necessitate that $B_{\delta/2}(z)$ is also in the collection, contradicting the maximality of $\{B_{\delta/2}(y_j)\}_{j=1,...,J}$. By (1), but replacing $\rho$ and $y$ with $\delta/2$ and $y_j$ and summing over $j$, we have
    \begin{align}
    J\delta^{n-2} \leq 2^n \epsilon^{-1} \int_{\cup B_{\delta}(y_j)} |Du|^2 \leq 2^n \epsilon^{-1} \int_{Q_\delta} |Du|^2
    \end{align}
    where $Q_\delta=\{x: \text{dist}(x,K \cap \text{sing}(u)\}$. Multiplying through by $\delta^2$, we have 
    \[
    J\delta^n \leq 2^n \delta^2 \epsilon^{-1} \int_{Q_{\delta_0}} |Du|^2
    \]
    since $\{B_{\delta/2}(y_j)\}_{j=1,...,J}$ cover all of $\text{sing}(u) \cap K$, we can let $\delta \downarrow 0$, and we obtain that $\lambda(\text{sing}(u) \cap K)=0$, where $\lambda$ is Lebesgue measure. Thus, $\int_{Q_\delta} |Du|^2 \rightarrow 0$ as $\delta \downarrow 0$ by the dominated convergence theorem. By (2), this implies that $H^{n-2}(\text{sing}(u) \cap K)=0$, and since $K$ was an arbitrary compact set of $\Omega$, we have that $H^{n-2}=0$ as required.
\end{proof}
As a consequence, sing(u) contains only isolated points, and is empty for $\Omega \subset \mathbb{R}^2$.
\begin{center}
    
\begin{tikzpicture}

    \draw[->] (-1,0,0) -- (4,0,0) node[anchor=north east]{$x$};
    \draw[->] (0,-1,0) -- (0,4,0) node[anchor=north west]{$y$};
    \draw[->] (0,0,-1) -- (0,0,4) node[anchor=south]{$z$};

    \fill (1,1,1) circle (2pt) node[anchor=south west]{$p_1$};
    \fill (2,2,0.5) circle (2pt) node[anchor=south west]{$p_2$};
    \fill (3,1.5,2) circle (2pt) node[anchor=south west]{$p_3$};
    \fill (0.5,3,1.5) circle (2pt) node[anchor=south west]{$p_4$};
    \fill (2.5,2.5,3) circle (2pt) node[anchor=south west]{$p_5$};
    \draw[dashed] plot [smooth cycle] coordinates {(0.1,3,1.5) (1,3.5,2) (4,4,4) (3.2,1.5,3) (2.5,1,0.8) (1.5,0.5,1) (0.8,0.8,0.5)};

\end{tikzpicture}
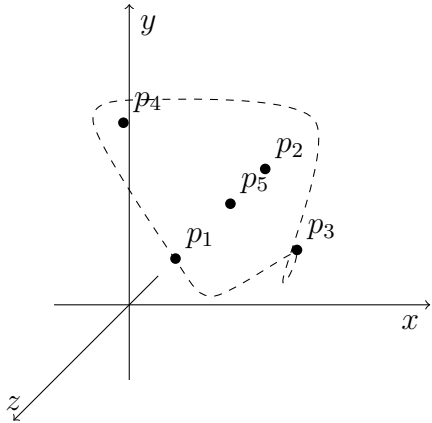
\\ Figure 3: sing(u) for a subset $\Omega$ in $\mathbb{R}^3$
\end{center}
\subsection{The Flag $\mathcal{S}_0 \subset \mathcal{S}_1 \subset ... \subset \mathcal{S}_{n-3} = \mathcal{S}_{n-2} = \mathcal{S}_{n-1} = \text{sing(u)}$}

Continuing from Section 2.5, for any $y \in \Omega$ and any tangent map $\phi$ of $u$ at $y$, $S(\phi)$ is the linear subspace of points $y \in \mathbb{R}^n$ such that $\Theta_\phi (y)=\Theta_\phi(0)$. From this we immediately obtain:
\[
y \in \text{sing}(u) \iff \dim S(\phi) \leq n-1 
\]
for every tangent map $\phi$ of $u$ at $y$.

\begin{definition}{11 (The set $\mathcal{S}_j$)}
    For each $j=0,1,...,n-1$, we define
    \[
    S_j = \{y \in \operatorname{sing} u : \dim S(\phi) \leq j \text{ for all tangent maps } \phi \text{ of } u \text{ at } y\}
    \]
\end{definition}

Here we see that by dimensionality, $\mathcal{S}_{n-1}=\text{sing}(u)$ follows immediately from the fact that $y \in \text{sing}(u) \iff \dim S(\phi) \leq n-1$. Now if $\mathcal{S}_{n-3}$ is not equal to both $\mathcal{S}_{n-2}$ and $\mathcal{S}_{n-1}$, then there exists a $y \in \text{sing}(u)$ where there is a tangent map $\phi$ with $\dim S(\phi)=n-1$ or $n-2$, but then $H^{n-2}(S(\phi))=\infty$ and since $S(\phi) \subset \text{sing}(u)$, we have that $H^{n-2}(\text{sing}(u))=\infty$, which contradicts the fact that $H^{n-2}(\text{sing}(u))=0$ by the previous section. 

In fact, this induces the \textbf{flag}, that is, an increasing sequence of subsets in $\mathbb{R}^n$, given by
\[
\mathcal{S}_0 \subset \mathcal{S}_1 \subset ... \subset \mathcal{S}_{n-3} = \mathcal{S}_{n-2} = \mathcal{S}_{n-1} = \text{sing(u)}
\]

\subsection{Further Geometric Properties of sing(u)}
The flag we introduced in the previous section is motivated by the following lemma, which has applications to minimal surface theory and GMT. This is a kind of "dimension reducing" argument, first pioneered by Herbert Federer, who is considered the father of GMT.

\begin{lemma}{2}
    For each $j=0,1,...,n-3$, $\dim \mathcal{S}_j \leq j$, and for each $\alpha > 0$, $\mathcal{S}_0 \cap \{ x: \Theta_u(x)=\alpha \}$ is a discrete set.
\end{lemma}
Here, $\dim$ refers to Hausdorff dimension, and thus $\dim \mathcal{S}_j \leq j \implies H^{j+\epsilon}(\mathcal{S}_j)=0$ for each $\epsilon>0$. Here, we will omit the proof that $\dim \mathcal{S}_j \leq j$, as that will require some heavy machinery from GMT, though I will include those lemmas needed to prove the result at the end.

\begin{proof}{($\mathcal{S}_0 \cap \{ x: \Theta_u(x)=\alpha \}$ is a discrete set)}
    Suppose that $\mathcal{S}_0 \cap \{ x: \Theta_u(x)=\alpha \}$ is not a discrete set for some $\alpha>0$. Then there are $y,y_j \in \mathcal{S}_0 \cap \{ x: \Theta_u(x)=\alpha \}$ such that $y_j \neq y$ for each $j$, and $y_j \rightarrow y$. Now let $\rho_j=|y_j-y|$, and consider the scaled maps $u_{y,\rho_j}$. By compactness, there exists a subsequence $\rho_{j'}$ such that $u_{y,\rho_{j'}} \rightarrow \phi$, and $\phi$ is definitionally the tangent map of $u$ at $y$. Moreover, we have that $\Theta_{\phi}(0)=\Theta_{u}(y)=\alpha$.
    Now define $\xi_j=|y_j-y|^{-1}(y_j-y) \in S^{n-1}$. Suppose that the subsequence $j'$ from above is such that $\xi_{j'}$ converges to some $\xi \in S^{n-1}$. Note that the transformation $x \mapsto y+\rho_j x$ takes $y_j$ to $\xi_j$, and hence we have that
    \[
    \Theta_u(y_j)=\Theta_{u_{y,\rho_j}}(\xi_j)=\alpha
    \]
    for each $j$. By the upper semi-continuity of the density function, we have that $\Theta_{\phi}(\xi) \geq \alpha$. Therefore, $\Theta_{\phi}(x)$ attains its maximum at 0, and so we have $\Theta_\phi(\xi)=\Theta_{\phi}(0)=\alpha$. Thus, $\xi \in S(\phi)$, contradicting the fact that $S(\phi)=\{ 0 \}$ since we assumed that $y \in \mathcal{S}_0$.
\end{proof}

\begin{corollary}{3}
    $\dim \text{sing}(u) \leq n-3$, and if $N$ is a 2-dimensional surface of genus $\geq 1$, then $\dim \text{sing}(u) \leq n-4$. More generally, if all tangent maps $\phi \in W_{\text{loc}}^{1,2}(\mathbb{R}^m; N)$ of $u$ satisfy $\dim S(\phi) \leq m$, then $\dim \text{sing}(u) \leq m$.
\end{corollary}

\begin{remark}
    The above corollary implies that $\dim \text{sing}(u) \leq m$ for every locally energy minimizing map $u \in W^{1,2}(\Omega;N)$ if $N$ is such that all homogeneous degree zero locally energy minimizing maps $\phi \in W^{1,2}(\mathbb{R}^n; N)$ satisfy $\dim S(\phi) \leq m$. 
    For example, if $\dim N = 2$ and $N$ has genus $g \geq 1$, we claim that this holds for $m=n-4$. Indeed suppose there is a homogeneous degree zero locally energy minimizing map $\phi$ of $\mathbb{R}^n \rightarrow N$ with $\dim S(\phi)=n-3$. WLOG, we can assume $S(\phi)=\{ 0 \} \times \mathbb{R}^{n-3}$, then we have $\phi(x,y)=\phi_0(|x|^{-1}x)$ for $x \in \mathbb{R}^3 \setminus \{ 0 \}$ and $y \in \mathbb{R}^{n-3}$. Then $\phi_0$ is a nonconstant minimizing map from $S^2$ into $N$, but such maps are known not to exist by the work of Jürgen Jost (1984)\cite{jost1984harmonic}.
\end{remark}

\begin{proof}
    We know that $\text{sing}(u)=\mathcal{S}_{n-3}$, and hence by the lemma with $j=n-3$, we obtain precisely that $\dim \text{sing}(u) \leq n-3$ as claimed. If $\dim N =2$ with genus $\geq 1$, we claim that $\text{sing}(u)=\mathcal{S}_{n-4}$, which gives that $\dim \text{sing}(u) \leq n-4$.
    Lastly, if $\dim S(\phi) \leq m$ for all tangent maps $\phi$, then by definition $\mathcal{S}_m=\text{sing}(u)$, and we have that $\dim \text{sing}(u) \leq m$ by the lemma.
\end{proof}

Here, I would like to provide the two lemmas needed to complete the proof that $\dim \mathcal{S}_j \leq j$. I will not provide the proofs, but they are quite technical and can be found in Simon's book on energy minimizing maps.

\begin{lemma}{3}
    For each \( y \in S_j \), and each \( \delta > 0 \) there is an \( \epsilon > 0 \) (depending on \( u, y, \delta \)) such that for each \( \rho \in (0, \epsilon] \)

\[ \eta_{y,\rho} \left\{ x \in B_{\rho}(y) : \Theta_u(x) \geq \Theta_u(y) - \epsilon \right\} \subset \text{the } \delta\text{-neighbourhood of } L_{y,\rho} \]

for some \( j \)-dimensional subspace \( L_{y,\rho} \) of \( \mathbb{R}^n \) (see Figure 3.1).

\end{lemma}

\begin{lemma}{4}
    There is a function \( \beta : (0, \infty) \to (0, \infty) \) with \( \lim_{t \to 0} \beta(t) = 0 \) such that if \( \delta > 0 \) and if \( A \) is an arbitrary subset of \( \mathbb{R}^n \) having the property \( (*) \) above, then

\[ \mathcal{H}^{j + \beta(\delta)}(A) = 0. \]

The remainder of the proof is technical, yet standard in Geometric Measure Theory, and can be found in Section 3.4 of Simon (1996, pp. 54-57)\cite{simon1996theorems}

\end{lemma}

\subsection{Open Problems Pertaining to sing(u)}
There are several fascinating open problems relating to the singular set which draw inspiration from Geometric Analysis and GMT. Here are several based on the work of Leon Simon and Fanghua Lin (1991, pp. 169)\cite{lin1991singular}:

\begin{op}{1}
Is sing(u) real-analytic or merely smooth?
\end{op}

This problem lies at the intersection of GMT, PDEs, and the calculus of variations. Understanding the regularity of the singular set has profound implications for the qualitative behavior of PDEs, especially as it pertains to elliptic and parabolic equations. Smooth or real-analytic singular sets would provide a higher degree of predictability and the structure of these solutions. This also bears on the greater question of determining the nature of singularities in various physical and geometric contexts, such as material science, fluid dynamics and general relativity.

\begin{op}{2}
Consider sing(u), where u is an energy minimizing map from an n-dimensional domain to $S^2$. Is $H^{n-3}(sing(u)) < \infty$?
\end{op}

As it pertains to Geometric Analysis and the theory of harmonic maps, this problem is of great significance. This question addresses the size and structure of the singularities that can occur in such maps. If $H^{n-3}(sing(u))$ is indeed finite, this would suggest that these singularities are in fact confined to a relatively small subset of the domain, which bears on the regularity of solutions. Moreover, this would enhance the understanding of the geometry of energy minimizing maps onto the target manifold $S^2$ and the analytical properties of $u$.

\begin{op}{3}
Under what conditions is sing(u) unique for a given energy minimizing map u?
\end{op}

This is a crucial problem for understanding the stability and predictability of singularities. Uniqueness would imply that, under specified conditions, the formulation of singularities follows a deterministic pattern, which is salient for understanding the behavior of solutions to variational problems. This has significant consequences for the modeling of physical systems where singularities represent defects or discontinuities, such as in materials science or fluid dynamics. Such insights could also bridge the gap between local and global properties of solutions.

\end{document}